\documentclass{amsart}

\usepackage{amsmath,amsthm,amsfonts,amssymb}

\usepackage{graphicx}
\usepackage{epstopdf}

\newtheorem{theorem}{Theorem}
\newtheorem{lemma}[theorem]{Lemma}
\newtheorem{proposition}[theorem]{Proposition}
\newtheorem{corollary}[theorem]{Corollary}
\newtheorem{definition}[theorem]{Definition}

\theoremstyle{definition}

\newcommand{\C}{\ensuremath{\mathbb{C}}}%
\newcommand{\Z}{\mathbb{Z}}

\newcommand{\alb}{\mathsf{X}}

\newcommand{\arr}{\longrightarrow}

\begin{document}

\title{On topological full groups of $\mathbb{Z}^d$-actions}

\author{M. Chornyi, K. Juschenko, V. Nekrashevych}%


\begin{abstract}
We give new examples of simple finitely generated groups arising from actions of free abelian groups on the Cantor sets. As particular examples, we discuss groups of interval exchange transformations, and a group naturally associated with the Penrose tilings. Many groups in this class are amenable.
\end{abstract}

\maketitle

\section{Introduction}
The motivation of this paper is to enrich the class of {\it non-elementary amenable} groups.
A group $G$ is amenable if there exists a finitely additive translation invariant probability measure on all subsets of $G$. This definition was given by John von Neumann, \cite{von-neumann1}, in a response to Banach-Tarski, and Hausdorff paradoxes. He singled out the property of a group which forbids paradoxical actions.

The class of {\it elementary amenable groups}, denoted by  $EG$, was introduced by Mahlon Day in \cite{day:semigroups}, as the smallest class of groups that contain finite and abelian groups and is closed under taking subgroups, quotients, extensions and directed unions. The fact that the class of amenable groups is closed under these operations was already known to von Neumann, \cite{von-neumann1}, who noted at that at that time there was no known amenable group which did not belong to $EG$.

No substantial progress in understanding this class has been made until the~80s, when  Chou, \cite{C}, showed that all elementary amenable groups have either polynomial or exponential growth, and Rostislav Grigorchuk, \cite{grigorchuk:milnor_en} gave an example of a group with intermediate growth. Grigorchuk's group served as a starting point in developing the theory of groups with intermediate growth, all of them being non-elementary amenable. In the same paper Chou showed that every simple finitely generated infinite group is not elementary amenable. In \cite{JM} it was shown that the  topological full group of Cantor minimal system is amenable.  By the results of  Matui, \cite{Matui}, this group has a simple and finitely generated commutator subgroup, in particular, it is not elementary amenable. This was the first example of infinite simple finitely generated amenable group.

Currently there are only two sources of non-elementary amenable groups: groups acting on rooted trees and topological full groups of Cantor minimal systems. In \cite{j-trees}, the author gives a unified approach to non-elementary amenability of groups acting on rooted trees. Here we give more examples of non-elementary amenable groups coming from topological full groups of Cantor minimal systems.

\begin{theorem}
\label{main}
Consider a minimal faithful action of $\mathbb{Z}^d$ on a Cantor set conjugate to the action on a closed $\Z^d$-invariant subset of $\alb^{\Z^d}$ for some finite alphabet $\alb$. Then the commutator subgroup of the topological full group $[[\mathbb{Z}^d]]$ is finitely generated.
\end{theorem}

We describe an explicit generating set, and give a simple direct proof. Subsequent paper~\cite{nek:fullgr} gives a less direct proof of the finite generation of special subgroups $\mathsf{A}(G)$ of the topological full groups in the case of expansive action. The groups $\mathsf{A}(G)$ are closely related to the derived subgroups of the full groups (in particular, they coincide in the case of the actions of abelian groups), but precise relation is still not well understood. 

It was proved in \cite{matui2}, that the commutator subgroup of $[[\mathbb{Z}^d]]$ is simple. The topological full groups that correspond to interval exchange transformation group were studied in \cite{JMMS}. The authors prove that subgroups of {\it rank} equals to 2 are amenable. These groups can be realized as topological full groups of minimal action of $\mathbb{Z}^2$ on the Cantor set. Therefore, Theorem~\ref{main} in the combination with \cite{matui2}, \cite{JMMS} gives more examples of simple finitely generated infinite amenable groups, and thus by the result of Chou non-elementary amenable groups.  Matui,  \cite{matui2}, showed that $[[T_1, \ldots, T_m]] = [[T_1, \ldots, T_n]]$ implies $m=n$. In particular, this implies that the groups from the corollary are different from the one previously obtained in \cite{JM}.

In the last Section we associate a group $\mathcal{P}$ to the Penrose tiling. The main result is
\begin{theorem}
The derived subgroup of $\mathcal{P}$ is simple and finitely generated.
\end{theorem}

It is an open question to decide if the group $\mathcal{P}$ is amenable.

\section{A finite generating set of the derived subgroup  of $[[\Z^d]]$}

\begin{lemma}
\label{lem:free}
Let $G$ be an abelian group. If the action of $G$ on a Cantor set $\mathbf{C}$ is minimal and faithful, then it is free.
\end{lemma}

\begin{proof}
Suppose that $g\in G$ is a non-zero element. By faithfulness of the action, there exists $x\in\mathbf{C}$ such that $g(x)\ne x$. Then there exists a neighborhood $U$ of $x$ such that $g(U)\cap U=\emptyset$. Let $y\in\mathbf{C}$ be an arbitrary point. By minimality, there exists $h\in G$ such that $h(y)\in U$. Since $U$ and $g(U)$ are disjoint, the points $h(y)$ and $gh(y)$ are different.
Then $g(y)=h^{-1}gh(y)\ne h^{-1}h(y)=y$. It follows that $g$ has no fixed points in $\mathbf{C}$.
\end{proof}

Let us fix a minimal action of the free abelian group $\Z^d$ on a closed shift-invariant subset $\mathbf{C}\subset\alb^{\Z^d}$ of the full shift over a finite alphabet $\alb$. Then $\mathbf{C}$ is homeomorphic to the Cantor set. We use the additive notation for the group $\Z^d$. If $w:\Z^d\arr\alb$ is a point of $\alb^{\Z^d}$, then its image under the action of $g\in\Z^d$ is defined by the rule
\[g(w)(h)=w(h-g).\]
In other words, elements of $\alb^{\Z^d}$ are labelings of the points of $\Z^d$ by elements of $\alb$, and the elements $g\in\Z^d$ act by shifting all the labels by $g$. Alternatively, we may imagine the action of $g$ as the shift of the ``origin of coordinates'' in a given sequence $w$ by $-g$.

A \emph{patch} $\pi=(f, P)$ is a finite subset $P\subset\Z^d$ together with a map $f:P\arr\alb$. The set $P$ is called the \emph{support} of the patch. We say that an element $w\in\alb^{\Z^d}$ (a \emph{$\Z^d$-sequence})  \emph{contains} the patch $(f, P)$ if $w|_P=f$. The set $\mathcal{W}_\pi$ of all sequences containing a given patch $\pi$ is a clopen subset of $\alb^{\Z^d}$ called the \emph{cylindrical set} defined by the patch, and the set of all such clopen subsets forms a basis of topology on $\alb^{\Z^d}$, by definition.

We say that two patches $\pi_1=(f_1, P_1)$ and $\pi_2=(f_2, P_2)$ are \emph{compatible} if there exists sequence $w\in\mathbf{C}$ containing $\pi_1$ and $\pi_2$. In other words, the patches are compatible if the intersection of the associated cylindrical sets is non-empty. If the patches $\pi_1$ and $\pi_2$ are compatible then there \emph{union} $\pi_1\cup\pi_2$ is the patch $(f, P_1\cup P_2)$ where $f|_{P_1}=f_1$ and $f|_{P_2}=f_2$.
Note that in terms of the cylindrical sets, we have $\mathcal{W}_{\pi_1\cup\pi_2}=\mathcal{W}_{\pi_1}\cap\mathcal{W}_{\pi_2}$.

If $\pi=(f, P)$ is a patch, and $Q$ is a finite set containing $P$, then $\mathcal{W}_\pi$ is equal to the disjoint union of the sets $\mathcal{W}_{(\tilde f, Q)}$, where $\tilde f$ runs through the set of all maps $\tilde f:Q\arr\alb$ such that $\tilde f|_P=f$. Note that some of these sets may be empty (if the corresponding patch is not allowed for the elements of $\mathbf{C}$). 

If $\pi=(f, P)$ is a patch and $g\in\Z^d$, then we have $g(\mathcal{W}_{\pi})=\mathcal{W}_{\pi+g}$, where $\pi+g=(f+g, P+g)$, where $(f+g):(P+g)\arr\alb$ is given by $(f+g)(h)=f(h-g)$, in accordance with the definition of the action of $\Z^d$ on $\alb^{\Z^d}$.

\begin{lemma} 
\label{lem:incompatible}
Let $A\subset\Z^d$ be a finite set not containing zero. Then there exists $B\subset\Z^d$ such that for every $w\in \mathbf{C}$ and every $g\in A$ the patches $(w|_B, B)$ and $((w+g)|_B, B)$ are not compatible.
\end{lemma}

\begin{proof}
Define the following metric on $\alb^{\Z^d}$. The distance $|w_1-w_2|$ is equal to $2^{-R}$, where $R$ is the biggest number such that restrictions of $w_1$ and $w_2$ to the ball of radius $R$ in $\Z^d$ with center in $0$ (for example in the $\ell_\infty$ norm) coincide. The it is enough to prove that there exists $\epsilon$ such that $|g(w)-w|>\epsilon$ for all $g\in A$ and $w\in\mathbf{C}$. Namely, for every $\epsilon$ there exists a finite set $B\subset\Z^d$ such that for every $w\in\alb^{\Z^d}$ the set of all $u\in\alb^{\Z^d}$ such that $(u|_B, B)=(w|_B, B)$ is contained in the $\epsilon$-neighborhood of $w$.

Suppose that it is not true, i.e., that for every $\epsilon>0$ there exist $w$ and $g\in A$ such that $|g(w)-w|\le \epsilon$. Since $A$ is finite, this implies that there exists $g\in A$ and a sequence of points $w_n\in\mathsf{C}$ such that $|g(w_n)-w_n|\to 0$ as $n\to \infty$. Since $\mathbf{C}$ is homeomorphic to the Cantor set, this implies that $g$ has a fixed point, which is a contradiction Lemma~\ref{lem:free}.
\end{proof}

Let $U\subset\mathbf{C}$ be a clopen set, and $g_1, g_2, g_3\in\Z^d$ elements such that $g_1(U), g_2(U), g_3(U)$ are pairwise disjoint. Denote by $T_{U, (g_1, g_2, g_3)}$ the element of $[[\Z^d]]$ given by
\[T_{U, (g_1, g_2, g_3)}(w)=\left\{\begin{array}{rl} (g_2-g_1)(w) & \text{if $w\in g_1(U)$;}\\
(g_3-g_2)(w) & \text{if $w\in g_2(U)$;}\\
(g_1-g_3)(w) & \text{if $w\in g_3(U)$;}\\
w & \text{if $w\notin g_1(U)\cup g_2(U)\cup g_3(U)$.}\end{array}\right.\]
In other words, $T_{U, (g_1, g_2, g_3)}$ cyclically permutes $g_1(U)$, $g_2(U)$, and $g_3(U)$ in the natural way. 
We will denote $T_{\pi, (g_1, g_2, g_3)}=T_{\mathcal{W}_\pi, (g_1, g_2, g_3)}$, for a patch $\pi$. 

\begin{lemma}
Let $A_1, A_2, A_3, B_1, B_2, B_3$ be subsets of a set $X$ such that only $A_1$ and $B_1$ have non-empty intersection, while all the other pairs of subsets are disjoint. Let $a$ be a permutation of $X$ acting trivially on $X\setminus (A_1\cup A_2\cup A_3)$, and satisfying $a(A_1)=A_2$, $a(A_2)=A_3$, $a(A_3)=A_1$, and $a^3=1$. Similarly, let $b$ be a permutation acting trivially on $X\setminus (B_1\cup B_2\cup B_3)$ and satisfying $b(B_1)=B_2$, $b(B_2)=B_3$, $b(B_3)=A_3$, and $b^3=1$. Then $[[b^{-1}, a^{-1}], [b, a]]$ acts as $a$ on the set $(A_1\cap B_1)\cup a(A_1\cap B_1)\cup a^2(A_1\cap B_1)$ and identically outside of it.
\end{lemma}

Note that we use the left action here, but the usual commutator $[g, h]=g^{-1}h^{-1}gh$.

\begin{proof}
Let $C=A_1\cap B_1$. Then the sets $C, a(C), a^2(C), b(C), b^2(C)$ are pairwise disjoint. The permutations $a$ and $b$ act as cycles of length three permuting $C, a(C), a^2(C)$ and $C, b(C), b^2(C)$, respectively. The element $[[b^{-1}, a^{-1}], [b, a]]$ acts then on these fives sets in the same way as the similar expression involving commutators of the permutations $a=(1, 2, 3)$ and $b=(1, 4, 5)$ act on $\{1, 2, 3, 4, 5\}$. We have $[b, a]=b^{-1}a^{-1}ba=(1, 5, 3)$ and $[b^{-1}, a^{-1}]=bab^{-1}a^{-1}=(1, 4, 2)$, and $[[b^{-1}, a^{-1}, [b, a]]=(1, 4, 2)^{-1}(1, 5, 3)^{-1}(1, 4, 2)(1, 5, 3)=(1, 2, 3)=a$.

The set $A'=(A_1\setminus C)\cup a(A_1\setminus C)\cup a^2(A_1\setminus C)$ is $a$-invariant, and $b$ acts trivially on it. It follows that the restriction of $[b^{-1}, a^{-1}]$ and $[b, a]]$ to $A'$ is equal to the restriction of $[1, a^{-1}]$ and $[1, a]$, which are trivial. It follows that $[[b^{-1}, a^{-1}], [b, a]]$ acts trivially on $A'$. The same argument shows that $[[b^{-1}, a^{-1}], [b, a]]$ acts trivially on $(B_1\setminus C)\cup b(B_1\setminus C)\cup b^2(B_1\setminus C)$.
\end{proof}

As a corollary, we get the following relation between the elements of the form $((f, P), g_1, g_2, g_3)$.

\begin{corollary}
\label{cor:patchesunion}
Let $\pi_1$, $\pi_2$ be patches, $g_1, g_2, h_1, h_2$ be elements of $\Z^d$ such that $\pi_1, \pi_1+g_1, \pi_1+g_2, \pi_2, \pi_2+h_1, \pi_2+h_2$ are pairwise incompatible except for the pair $\pi_1$ and $\pi_2$. Then
\[[[T_{\pi_2, (0, h_1, h_2)}^{-1}, T_{\pi_1, (0, g_1, g_2)}^{-1}], [T_{\pi_2, (0, h_1, h_2)}, T_{\pi_1, (0, g_1, g_2)}]]=T_{\pi_1\cup\pi_2, (0, g_1, g_2)}.\]
\end{corollary}

We will use the usual $\ell_1$ metric on $\Z^d$, i.e., the word metric associated with the standard generating set of $\Z^d$. Denote by $B(R)$ the ball of radius $R$ with the center in $0\in\Z^d$ for this metric.

By Lemma~\ref{lem:incompatible}, there exists $R_1$ such that for every $w\in\mathbf{C}$ the patches $\pi=(B(R_1), w|_{B(R_1)})$ and $\pi+g$ are incompatible for every $g\in\Z^d$ of length $\le 3$.

Let $\{e_1, e_2, \ldots, e_d\}$ be the standard generating set of $\Z^d$.
Denote by $\mathcal{T}_R$ the set of elements of $[[\Z^d]]'$ of the form $T_{\pi, (0, e_i, -e_i)}$, where $\pi$ runs through the set of all patches of the form $(B(R), w|_{B(R)})$ for $w\in\mathbf{C}$.

\begin{proposition}
If $R\ge R_1+2$, then the group generated by $\mathcal{T}_R$ contains $\mathcal{T}_{R+1}$.
\end{proposition}

\begin{proof}
Denote $S=\{\pm e_1, \pm e_2, \ldots, \pm e_d\}$.
Let $B\subset\Z^d$ be a finite subset containing $B(R_1+2)$. 
Define the patches $\rho_0=(w|_B, B)$, $\rho_h=(w|_{B+h}, B+h)$,  for $h\in S$. Note that the patch $\rho_h$ contains the patch $(w|_{B(R_1)}, B(R_1))$ and that $\rho_h-h=((w-h)|_B, B)$.

Let us apply Corollary~\ref{cor:patchesunion} for $\pi_1=\rho_0$, $\pi_2=\rho_h$,
$g_1=g$, $g_2=-g$, $h_1=h$, $h_2=2h$, where $g, h$ are different elements of $S$. Since $\rho_0$ and $\rho_h$ both contain the patch $(w|_{B(R_1)}, B(R_1))$, the patches $\pi_1, \pi_1+g_1$, $\pi_1+g_2$, $\pi_2$, $\pi_2+h_1$, and $\pi_2+h_2$ are pairwise incompatible, except for $\pi_1$ and $\pi_2$, which are both patches of $w$. It follows that we can apply Corollary~\ref{cor:patchesunion}, hence
\[[[T_{\rho_h, (0, h, 2h)}^{-1}, T_{\rho_0, (0, g, -g)}^{-1}], [T_{\rho_h, (0, h, 2h)}, T_{\rho_0, (0, g, -g)}]]=T_{\rho_0\cup \rho_h, (0, g, -g)}.\]
Note that $T_{\rho_h, (0, h, 2h)}=T_{\rho_h-h, (-h, 0, h)}=T_{\rho_h-h, (0, h, -h)}$.

Let us apply now Corollary~\ref{cor:patchesunion} to $\pi_1=\rho_0-g$, $\pi_2=\rho_g-g$, $g_1=-2g$, $g_2=-g$, $h_1=h$, $h_2=-h$. The patches $\pi_1$ and $\pi_2$ are patches of $w-g$, and contain $((w-g)|_{B(R_1)}, B(R_1))$. It follows that, in the same way as above, we can apply Corollary~\ref{cor:patchesunion}, and get
\[[[T_{\rho_g-g, (0, h, -h)}^{-1}, T^{-1}_{\rho_0-g, (0, -2g, -g)}], 
[T_{\rho_g-g, (0, h, -h)}, T_{\rho_0-g, (0, -2g, -g)}]]=T_{(\rho_0-g)\cup(\rho_h-g), (0, -2g, -g)}.\]
Recall that $\rho_g-g=((w-g)|_B, B)$ and that we have $T_{\rho_0-g, (0, -2g, -g)}=T_{\rho_0, (g, -g, 0)}=T_{\rho_0, (0, g, -g)}$. We also have $T_{(\rho_0-g)\cup (\rho_h-g), (0, -2g, -g)}=T_{\rho_0\cup\rho_h, (0, g, -g)}$.

We have shown that the group generated by the set
\[\{T_{\pi, (0, g, -g)}\;:\;g\in S, \pi=(w|_B, B), w\in\mathbf{C}\}\]
contains the set
\[\{T_{\pi, (0, g, -g)}\;:\;g\in S, \pi=(w|_{B\cup B+h}, B\cup B+h), w\in\mathbf{C}, h\in S\}.\]
Since $B(R+1)=\bigcup_{h\in S}B(R)+h$, this finishes the proof of the proposition.
\end{proof}

It follows that the group generated by $\mathcal{T}_{R_1+2}$ contains $\mathcal{T}_R$ for every $R\ge R_1+2$. For every cylindrical set $U\subset\mathbf{C}$ there exists $R$ such that $U$ is equal to the disjoint union of cylindrical sets $\mathcal{W}_\pi$ such that $\pi$ is a patch with support $B(R)$. It follows that every element of the form $T_{\pi, (0, e_i, -e_i)}$ can be written as a product of elements of the same type such that $\pi$ is a patch with support $B(R)$ for some $R$ big enough. Consequently, the group generated by $\mathcal{T}_{R_1+2}$ contains all elements of the form $T_{\pi, (0, e_i, -e_i)}$.

The proof of Theorem~\ref{main} is finished by the following, since the set $\mathcal{T}_{R_1+2}$ is finite.

\begin{proposition}
The derived subgroup of the full group of the action of $\Z^d$ on $\mathsf{C}$ is generated by the set of elements of the form $T_{\pi, (0, e_i, -e_i)}$.
\end{proposition}

\begin{proof}
It is known, see~\cite{matui1}, that the derived subgroup of $[[\Z^d]]$ is simple and is contained in every non-trivial normal subgroup of $[[\Z^d]]$. 

Consider the set $\mathcal{T}\subset[[\Z^d]]$ of all elements elements of order three  permuting cyclically three disjoint clopen subsets $U_1, U_2, U_3$ of $\mathbf{C}$ and acting identically outside their union. The set $\mathcal{T}$ is obviously invariant under conjugation by elements of $[[\Z^d]]$, hence the group generated by $\mathcal{T}$ is normal. On the other hand, we have $\mathcal{T}\subset [[\Z^d]]'$, as every element of $\mathcal{T}$ is equal to the commutator of two transformations: one permuting $U_1$ with $U_2$, and the other permuting $U_2$ with $U_3$. Consequently, $\mathcal{T}$ generates $[[\Z^d]]'$.

For every element $T\in\mathcal{T}$ permuting cyclically clopen sets $U_1, U_2, U_3$, there exists partitions of $U_i$ into cylindrical sets such that $T$ maps a piece of the partition to a piece of the partition, and restriction of $T$ to every piece of the partitions is equal to the restriction of an element of $\Z^d$. It follows that $T$ is a product of a finite set of elements of the form $T_{\pi, (g_1, g_2, g_3)}$. It remains to show that we can generate all elements of the form $T_{\pi, (g_1, g_2, g_3)}$ by elements of the form $T_{\pi, (0, e_i, -e_i)}$. It is well known that the alternating group $A_n$ is generated by cycles $(k, k+1, k+2)$. It follows that the group generated by $T_{\pi, (0, e_i, -e_i)}$ contains the set of elements of the form $T_{\pi, (g_1, g_2, g_3)}$, where $g_i$ belong to one direct factor of $\Z^d$. 

Let us prove the following technical lemma.

\begin{lemma}\label{l1}
Let $X_d=\{x_1\ldots x_d|x_i \in \{a,b,c\}, \ 1\le i \le d \}$ be the $3^d$-element set of $d$-letter words over the alphabet $\{a,b,c\}$, and let $S_{X_d}$ be the symmetric group of permutations of $X_d$. Denote the alternating subgroup of even permutations of $X_d$ by $A_{X_d}$. Consider the set $B_d$ of all elements of the type $(XaY\ XbY\ XcY) \in S_{X_d}$, where $X$ and $Y$ are arbitrary (possibly, empty) words such that $|X|+|Y|=d-1$. Then $A_{X_d}$ is generated by the set $B_d$.
\end{lemma}

\begin{proof}
The lemma can be proved by induction on $d$.

For $d=2$, we use the well-known fact that $A_9$ is generated by 3-cycles $\{(1 2 3), (2 3 4), \ldots, (7 8 9)\}$. To apply this fact, we need to show that all 7 elements $(aa \ ab \ ac)$, $(ab \ ac \ ba)$, $(ac \ ba \ bb)$, $(ba \ bb \ bc)$, $(bb \ bc \ ca)$, $(bc \ ca \ cb)$, $(ca \ cb \ cc)$ are generated by $B_2$. This can be checked by hand:

\begin{itemize}
	\item $(aa \ ab \ ac) \in B_d$;
	\item $(ab \ ac \ ba) = (aa \ ba \ ca)(aa \ ab \ ac)(aa \ ca \ ba) $;
	\item $(ac \ ba \ bb) = (ac \ cc \ bc)(ba \ bb \ bc)(ac \ bc \ cc) $;
	\item $(ba \ bb \ bc) \in B_d$;
	\item $(bb \ bc \ ca) = (aa \ ba \ ca)(ba \ bb \ bc)(aa \ ca \ ba) $;
	\item $(bc \ ca \ cb) = (ac \ cc \ bc)(ca \ cb \ cc)(ac \ bc \ cc) $;
	\item $(ca \ cb \ cc) \in B_d$.
\end{itemize}

Suppose the statement holds for $d=k$ and consider the case $d=k+1$. Since the alternating group is generated by 3-cycles, it's sufficient to show that every 3-cycle is generated by $B_{k+1}$. Assume we have a cycle $(Ax \ By \ Cz)$, where $A,B,C \in X_k$ are pairwise distinct, and $x,y,z \in \{a,b,c\}$, not necessarily distinct. We know the following:

\begin{itemize}
	\item $(Ax \ Bx \ Cx)$, $(Ay \ By \ Cy)$, $(Az \ Bz \ Cz)$ are generated by $B_{k+1}$. Indeed, we can take the elements of $B_k$ generating $(A \ B \ C)$ and append the needed letter to each of them.
	\item $(Ax \ Ay \ Az)$, $(Bx \ By \ Bz)$, $(Cx \ Cy \ Cz)$ are in $B_{k+1}$ by definition.
\end{itemize}

Then, applying the induction base for the set $\{A,B,C\} \times \{a,b,c\}$, we conclude that $(Ax \ By \ Cz)$ is also generated by $B_{k+1}$.

In case $A,B,C$ are not distinct, we can use a slightly modified version of the proof above. If, for example, $A=B$ (which automatically implies $x \neq y$), we can take an arbitrary word $D \in X_k$ distinct from $A$ and $C$ in order to apply the induction base to $\{A,C,D\} \times \{a,b,c\}$. Clearly, the 3-cycle $(Ax \ Ay \ Cz)$ will belong to $X_{k+1}$



The induction step is complete.
\end{proof}

Lemma~\ref{l1} implies that the group generated by all elements of the form $T_{\pi, (g_1, g_2, g_3)}$, where $g_i$ belong to one factor of $\Z^d$, contains all elements of the form $T_{\pi, (g_1, g_2, g_3)}$, where $g_i\in\Z^d$ are now arbitrary. This finishes the proof of the proposition and Theorem~\ref{main}.
\end{proof}

\section{Topological full group and interval exchange group}

Let $\alpha_1, \alpha_2, \ldots, \alpha_d$ be irrational numbers such that the additive group $H=\langle\alpha_1, \alpha_2, \ldots, \alpha_d\rangle/\mathbb{Z}$ generated by them modulo $\mathbb{Z}$ is isomorphic to $\mathbb{Z}^d$ (this implies that $\langle\alpha_1, \alpha_2, \ldots, \alpha_d\rangle$ is also isomorphic to $\mathbb{Z}^d$). The group $H$ is a subgroup of the circle $\mathbb{R}/\mathbb{Z}$, and hence acts on it in the natural way. By the classical Kroneker's theorem, the action of each subgroup $\langle\alpha_i\rangle$ on $\mathbb{R}/\mathbb{Z}$ is minimal, hence the action of $H$ on $\mathbb{R}/\mathbb{Z}$ is also minimal.

Let us lift $H$ as a set to $[0, 1]$ by the natural quotient map $[0, 1]\to\mathbb{R}/\mathbb{Z}$, and let $W\subset[0, 1]$ be the obtained set.

Let us replace each number $q\in W\subset[0, 1]$ by two copies: $q_{-0}$ and $q_{+0}$. Here we identify $0_{-0}$ with $1$ and $0_{+0}$ with $0$, $1_{-0}$ with 1 and $1_{+0}$ with 0, according to the natural cyclic order on $\mathbb{R}/\mathbb{Z}$ (seen also as the quotient of the interval $[0, 1]$). Denote by $R_H$ the obtained set (equal to the disjoint union of $[0, 1]\setminus W$ and the set of doubled points $W$). The set $R_H$ is ordered in the natural way (we assume that $q_{-0}<q_{+0}$), and the order is linear (total).

Let us introduce the order topology on $R_H$. Recall, that it is the topology generated by the open intervals $(a, b)=\{x\in R_H\;:\;a<x<b\}$.

\begin{lemma}
The space $R_W$ is homeomorphic to the Cantor set.
\end{lemma}

\begin{proof}
	We use the following formulation of Brouwer's theorem: A topological space is a Cantor space if and only if it is non-empty, compact, totally disconnected, metrizable and has no isolated points. Note that by classical metrization theorems, we can replace metrizability by Hausdorffness and second countability.
	
The space $R_H$ is obviously non-empty, has no isolated points. For any $a, b\in W\cap[0, 1]$ such that $a<b$, we have $[a_{+0}, b_{-0}]=(a_{-0}, b_{+0})$, hence the intervals $(a_{-0}, b_{+0})$ are clopen. The set of such intervals is a basis of topology, since the set $W$ is dense. We also see that the space $R_H$ is second countable and Hausdorff.

Let $A\subset R_H$ be an arbitrary subset. Let us show that $\sup A$ and $\inf A$ exist, which will imply compactness. Let $\hat A$ be the image of $A$ in $[0, 1]$. We know that $\sup\hat A, \inf\hat A\in [0, 1]$ exist. If $\sup\hat A\notin W$, then the corresponding element of $R_H$ is also a supremum of $A$. If $\sup\hat A\in W$, then $\sup A=\sup\hat A_{-0}$, unless $\sup\hat A_{+0}\in A$, in which case $\sup A=\sup\hat A_{+0}$. Infima are treated in the same way.
\end{proof}

The action of $H$ on $\mathbb{R}/\mathbb{Z}$ naturally lifts to an action on $R_H$: we just set $h(q_{+0})=h(q)_{+0}$ and $h(q_{-0})=h(q)_{-0}$.

Denote by $IET_H$ the topological full group of the action $(H, R_H)$. For every element $g\in IET_H$ there exists a finite partition of $R_H$ into clopen subsets such that the action of $g$ on each of the subsets coincides with a translation by an element of $H$. Clopen subsets of $R_H$ are finite unions of intervals of the form $(a_{+0}, b_{-0})$ for $a, b\in H$. It follows that $g$ is an \emph{interval exchange transformation}: it splits the interval $[0, 1]$ into a finite number of intervals and then rearranges them. The endpoints of the intervals belong to $W$. Conversely, every interval exchange transformation such that the endpoints of the subintervals belong to $W$ is lifted to an element of $IET_H$.

We have proved the following.

\begin{lemma}
The group $IET_H$ is naturally isomorphic to the group of all interval exchange transformations of $[0, 1]$ such that the endpoints of the intervals into which $[0, 1]$ is split belong to $H$.
\end{lemma}

Theorem~\ref{main} now implies the following.

\begin{theorem}
The derived subgroup of $IET_H$ is simple and finitely generated.
\end{theorem}

A two-dimensional version of an interval exchange transformation group is considered in the next section.

\section{Penrose tiling group}

There are several versions of the Penrose
tiling~\cite{penrose}, let us describe one of them. The tiles
are two types of rhombi of equal side length 1. The angles of one
rhombus are $72^\circ$ and $108^\circ$. The angles of the other are
$36^\circ$ and $144^\circ$. We call these rhombuses ``thick'' and
``thin'', respectively. Mark a vertex of angle $72^\circ$ in the thick
rhombus, and a vertex of angle $144^\circ$ of the thin rhombus. Mark
the sides adjacent to the marked vertex by single arrows pointing towards
the marked vertex. Mark the other edges by double arrows, so that in the
thick rhombus they point away from the unmarked vertex of angle
$72^\circ$ and in the thin rhombus they point towards the unmarked
vertex of angle $144^\circ$, see Figure~\ref{fig:penrosetiling}.
A \emph{Penrose tiling} is a tiling of
the whole plane by
such rhombi, where markings of the edges match (adjacent tiles must
have same number of arrows pointing in the same direction). See
Figure~\ref{fig:tiling} for an example of a patch of a Penrose tiling.

\begin{figure}
\centering
\includegraphics{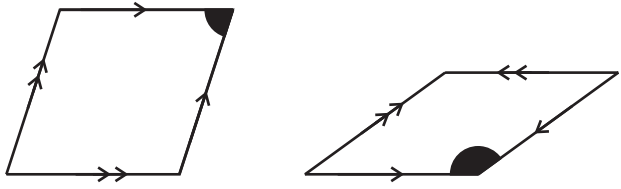}
\caption{Tiles of the Penrose tilings}
\label{fig:penrosetiling}
\end{figure}

\begin{figure}
\centering
\includegraphics{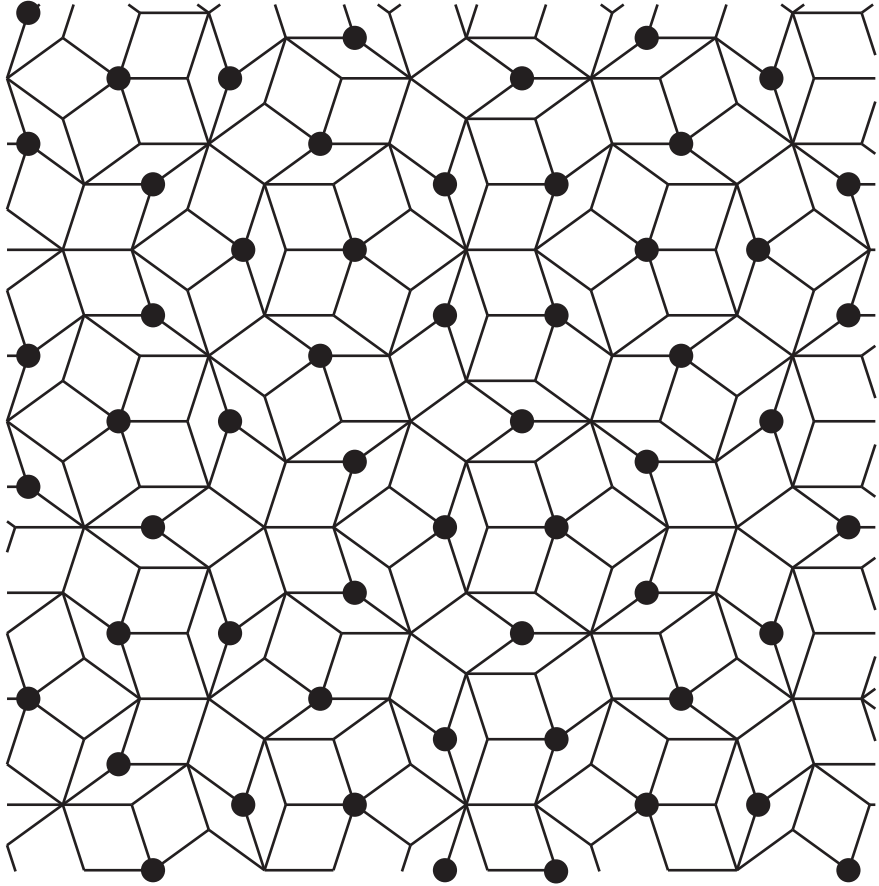}
\caption{Penrose tiling}
\label{fig:tiling}
\end{figure}

There are uncountably many different (up to translation and rotation)
Penrose tilings. Each of them is \emph{aperiodic}, i.e., does not
admit a translational symmetry.

Let us identify $\mathbb{R}^2$ with $\mathbb{C}$, and consider all
Penrose tilings by rhombi such that their sides are parallel to the
lines $e^{k\pi i/5}\mathbb{R}$, $k\in\mathbb{Z}$. A \emph{pointed}
Penrose tiling is a Penrose tiling with a marked vertex of a tile. Let
$\mathcal{T}$ be the set of all such pointed Penrose tilings, up to
translations (two pointed tilings correspond to the same element of
$\mathcal{T}$ if and only if there exists a translation mapping one
tiling to the other and the marked vertex of one tiling to the marked
vertex of the other). We sometimes identify a tiling with the set of vertices of
its tiles.

Let us introduce a topology on $\mathcal{T}$ in the following way. Let
$A\subset T$ be a finite  set of vertices of a Penrose tiling (a \emph{patch}), and let
$v\in A$. The corresponding open set
$\mathcal{U}_{A, v}$ is the set of all pointed tilings $(T, u)$ such that
$A+u-v\subset T$. In other words, a pointed
tiling $(T, u)$ belongs to $\mathcal{U}_{A, v}$ if we can see the pointed
patch $(A, v)$ around $u$ as a part of $T$. Then the
natural topology on $\mathcal{T}$ is given by the basis of open sets
of the form $\mathcal{U}_{A, v}$ for all finite pointed patches $(A,
v)$ of Penrose tilings. It follows from the properties of Penrose
tilings that the space $\mathcal{T}$ is homeomorphic to the Cantor
set, and that for every Penrose tiling $T$ the set of pointed tilings
$(T, v)$ is dense in $\mathcal{T}$. In the literature, see..., the
space $\mathcal{T}$ is called sometimes \emph{transversal}.

Consider  a patch $A$ with two marked vertices $v_1, v_2\in A$. Then we have a
natural homeomorphism
$F_{A, v_1, v_2}:\mathcal{U}_{A, v_1}\arr\mathcal{U}_{A, v_2}$
mapping $(T, u)\in\mathcal{U}_{A, v_1}$ to $(T,
u+v_2-v_1)\in\mathcal{U}_{A, v_2}$. The homeomorphism $F_{A, v_1,
  v_2}$ moves in every patch $A$ the marking from the vertex $v_1$ to
the vertex $v_2$. It is easy to see that $F_{A, v_1, v_2}$ is a
homeomorphism between clopen subsets of $\mathcal{T}$.

\begin{definition}
The \emph{topological full group of Penrose tilings} is the group
$\mathcal{P}$ of
homeomorphisms of $\mathcal{T}$ that are locally equal to the
homeomorphisms of the form $F_{A, v_1, v_2}$.
\end{definition}

The set of all pointed tilings $(T, v)$ obtained from a given
tiling $T$ is dense in $\mathcal{T}$ and invariant under the action of
the topological full group. It follows that every element of the full
group is uniquely determined by the permutation it induces on the set
of vertices of the tiling. In terms of permutations of $T$ the full
group can be defined in the following way.

We say that a map
$\alpha:T\arr T$ \emph{is defined by local rules} if there exists $R$
such that for every $x\in T$ the value of $x-\alpha(x)$ depends only
on the set $B_R\cap (T-x)$, where $B_R$ is the disc of radius $R$
around the origin $(0, 0)\in\mathbb{R}^2$.

The following is straightforward.

\begin{proposition}
A permutation $\alpha:T\arr T$ is induced by the element of the full
group if and only if $\alpha$ is defined by a local
rule. Consequently, the full group is isomorphic to the group of all
permutations of $T$ defined by local rules.
\end{proposition}

Let us describe a more explicit model of the space $\mathcal{T}$ and
the full group $\mathcal{P}$ using a description of the Penrose tilings given in
the papers~\cite{bruijn:pen1,bruijn:pen2}.

Denote $\zeta=e^{\frac{2\pi i}5}$, and let
\[P=\left\{\sum_{j=0}^4 n_j\zeta^j\;:\;n_j\in\mathbb{Z}, \sum_{j=0}^4
  n_j=0\right\}=
(1-\zeta)\mathbb{Z}[\zeta]\]
be the group generated by the vectors on the sides of the regular
pentagon $S=\{1, \zeta, \zeta^2, \zeta^3, \zeta^4\}$. Note that
$5=4-\zeta-\zeta^2-\zeta^3-\zeta^4\in P$.
As an abelian group, $P$ is isomorphic to $\mathbb{Z}^4$.

Denote by $\mathcal{L}$ the set of lines of the form
$i\zeta^j\mathbb{R}+w$, for $j=0, 1, \ldots, 4$ and $w\in P$. It is
easy to see that for any two intersecting lines $l_1,
l_2\in\mathcal{L}$ and any generator $z\in\{1-\zeta, \zeta-\zeta^2,
\zeta^2-\zeta^3, \zeta^3-\zeta^4\}$ there exists $z'\in P$ such that
$z'$ is parallel to $l_2$ and $l_1+z=l_1+z'$. It follows that for any
pair of intersecting lines $l_1, l_2\in\mathcal{L}$ the intersection
point $l_1\cap l_2$ belongs to $P$. Consequently, a point $\xi\in\C$
belongs either to 0, 1, or to 5 lines from $\mathcal{L}$. If $\xi\in\C$
does not belong to any line $l\in\mathcal{L}$, then we call $\xi$
\emph{regular}.

Similarly to the case of interval exchange transformations, let us double each line $l\in\mathcal{L}$.
Let $\mathcal{C}$ be the obtained space and let
$Q:\mathcal{C}\arr\mathbb{C}$ be the corresponding quotient map. If
$\xi\in\mathbb{C}$ is regular, then $Q^{-1}(\xi)$ consists of a single
point. If $\xi\in\mathcal{C}\setminus P$ belongs to a line
$l\in\mathcal{L}$,
then $Q^{-1}(\xi)$ consists of two points associated with each of the
two half-planes into which $l$ separates $\mathbb{C}$. Every point
$\xi\in P$ has 10 preimages in $\mathcal{C}$ associated with each of
the ten sectors into which the lines from $\mathcal{L}$ passing
through $\xi$ separate the plane. A sequence $\xi_n$ of points of $\mathcal{C}$
converges to a point $\xi\in\mathcal{C}$ if and only if the sequence $Q(\xi_n)$
converges to $Q(\xi)$ and the sequence $\xi_n$ eventually belongs (if $Q(\xi)$ is not
regular) to the associated closed half-plane or sector.
The space $\mathcal{C}$ is locally compact and totally
disconnected. Polygons with sides belonging to lines
from $\mathcal{L}$ form a basis of topology of $\mathcal{C}$.

The group $P$ acts on $\mathcal{C}$ in the natural way, so that the
action is projected by $Q$ to the action of $P$ on $\mathbb{C}$ by
translations. Therefore, sums of the form $\tilde\xi+a$, for
$\tilde\xi\in\mathcal{C}$ and $a\in P$, are well defined.

Let us describe, following~\cite{bruijn:pen1,bruijn:pen2}, how a Penrose tiling is
associated with a point $\tilde\xi\in\mathcal{C}$. We will usually
denote $\xi=Q(\tilde\xi)$.
Suppose that $\xi$ is regular.
The vertices of the corresponding tiling $T_{\tilde\xi}$ will be the points of the
form $\sum_{j=0}^4k_j\zeta^j$, where $k_j\in\mathbb{Z}$ are such that
\[\left(\sum_{j=0}^4 k_j,\quad\sum_{j=0}^4 k_j\zeta^{2j}+\xi\right)\in
\bigcup_{s=1}^4(s, V_s),\] where  $V_1$
is the pentagon with vertices $\zeta^j$, $V_2$ is the pentagon with
vertices $\zeta^j+\zeta^{j+1}$, $V_3=-V_2$, and $V_4=-V_1$. (Note that
we have changed $\xi$ to $-\xi$ comparing with~\cite{bruijn:pen1,bruijn:pen2}.)

If $\xi$ is singular, then we
can find a sequence $\tilde\xi_n$ of regular points converging in $\mathcal{C}$ to
$\tilde\xi$, and then the tiling $T_{\tilde\xi}$ is the limit of the
tilings $T_{\tilde\xi_n}$.

 Let $v=\sum_{j=0}^4n_j\zeta^{2j}\in P$
and $v'=\sum_{j=0}^4n_j\zeta^j$. Then $x\in T_{\tilde\xi}$ if and only if
$x-v'\in T_{\tilde\xi+v}$. It follows that action of $P$ on
$\mathcal{C}$ preserves the associated tilings up to translations. In
fact, it is not hard to show that
two tilings $T_{\tilde\xi_1}$ and $T_{\tilde\xi_2}$ are
translations of each other if and only if $\tilde\xi_1$ and
$\tilde\xi_2$ belong to one $P$-orbit, see~\cite{bruijn:pen1,bruijn:pen2}....

Note that sides of the pentagons $V_s'=V_s-s$ are contained in lines from the
collection $\mathcal{L}$, hence they are naturally identified with
compact open subsets of $\mathcal{C}$. See Figure~\ref{fig:vs} for the
pentagons $V_s'$. Denote by $V'=\bigcup_{s=1}^4 (s, V_s')$.

\begin{figure}
\centering
\includegraphics{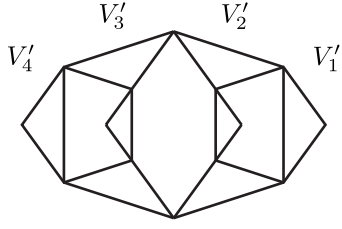}
\caption{Pentagons $V_s'$}
\label{fig:vs}
\end{figure}

For every $(s, \tilde\xi)\in V'$ the point
$s=s+0\cdot\zeta+0\cdot\zeta^2+0\cdot\zeta^3+0\cdot\zeta^4$ belongs to
the tiling $T_{\tilde\xi}$. We say that the pointed tiling
$(T_{\tilde\xi}, s)$ \emph{corresponds} to the point $(s,
\tilde\xi)\in V'$.

Let $x=\sum_{j=0}^4k_j\zeta^j\in T_{\tilde\xi}$, and let
$s=\sum_{j=0}^4k_j$. Then the numbers $v'=x-s$ and
$v=\sum_{j=0}^4k_j\zeta^{2j}-s$ belong to $P$, and the map $y\mapsto
y-v'$ is a bijection $T_{\tilde\xi}\arr T_{\tilde\xi+v}$. This maps
moves $x$ to the marked vertex $s=x-(x-s)$ of the tiling corresponding to $(s,
\tilde\xi+v)$. It follows that every pointed tiling, up to translation,
corresponds to a point of $V'$. It is easy to see that every
pointed Penrose tiling is represented by a unique point of $V'$, so
that we get a bijection between $V'$ and the space $\mathcal{T}$. It follows from the results
of~\cite{bruijn:pen1,bruijn:pen2} that this bijection is a homeomorphism.

\begin{proposition}
\label{pr:fullVprime}
The group $\mathcal{P}$ acts on $V'\cong\mathcal{T}$ locally by
translations by elements of $P$. In other words, for every
$\alpha\in\mathcal{P}$ there exists a partition of $V'$ into disjoint
clopen subsets $(s_i, U_i)$ such that $\alpha$ acts on each of them by
a translation $\alpha(s_i, x)=(s_i', x+\xi_i)$ for some $s_i'\in\{1,
2, 3, 4\}$ and $\xi_i\in P$.
\end{proposition}

Let us find some elements $t_i\in P$ such that $V_s'+t_i$ are pairwise
disjoint, and denote by $V''\subset\mathcal{C}$ the union of the sets
$V_s''=V_s'+t_i$. Then it follows from Proposition~\ref{pr:fullVprime}
that $\mathcal{P}$ is the group of all transformations $V''\arr V''$
that are locally equal to translations by elements of $P$.

Let us say that two clopen sets $U_1, U_2\subset\mathcal{C}$ are
\emph{equidecomposable} if there exists a homeomorphism $\phi:U_1\arr
U_2$ locally equal to translations by elements of $P$. If $U$ is any
clopen subset which is equidecomposable with $V''$, then $\mathcal{P}$
is equal to the group of all transformations of $U$ that are locally
translations by elements of $P$.

\begin{proposition}
The set $V''$ is equidecomposable with the parallelogram $F$ with vertices
$0, w_1=\zeta^2-\zeta^3, w_2=5(1-\zeta^2-\zeta^3+\zeta^4), w_1+w_2$.
\end{proposition}

\begin{proof}
Let us cut the pentagons $V_s''$ into triangles as it is shown on
Figure~\ref{fig:pentagonscut}.

\begin{figure}
\centering
\includegraphics{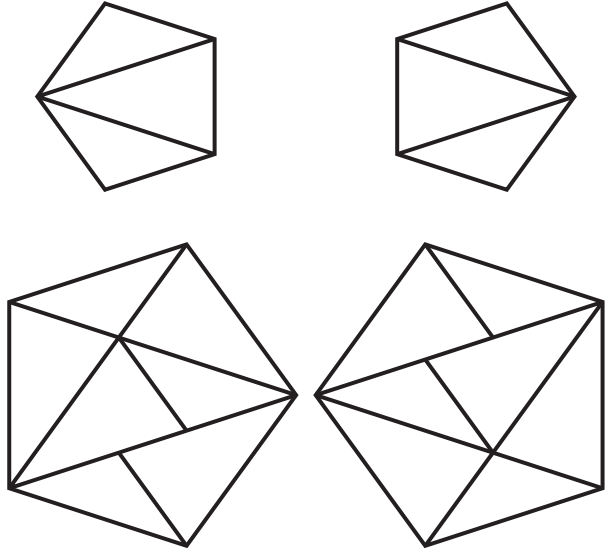}
\caption{Cutting pentagons $V_s$ into triangles}
\label{fig:pentagonscut}
\end{figure}

The obtained triangles can be grouped into pairs of triangles $T, T'$ such
that $T'$ is obtained from $T$ by a rotation by $2\pi$ (and
translation). Such pairs can be put together to form parallelograms,
as it is shown on Figure~\ref{fig:parallelograms}.

\begin{figure}
\centering
\includegraphics{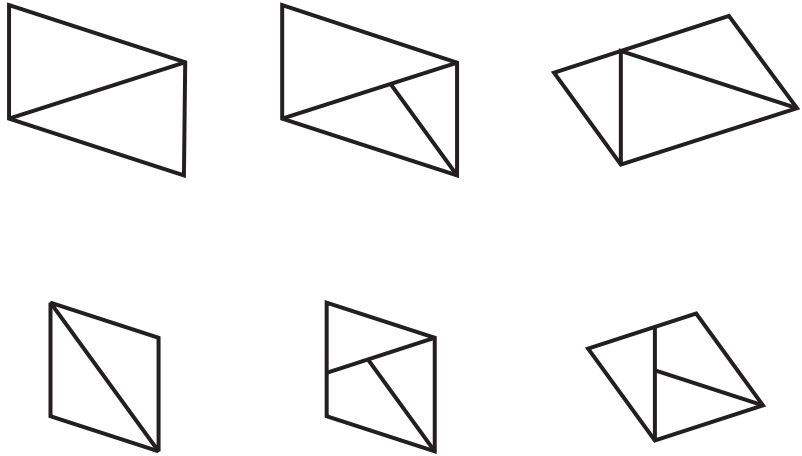}
\caption{Equidecomposability of parallelograms}
\label{fig:parallelograms}
\end{figure}

Figure~\ref{fig:parallelograms} also shows that each such
parallelogram is equidecomposable with its rotation by $\pi/5$. It
follows that each parallelogram is equidecomposable with its rotation
by any angle of the form $k\pi/5$.  Consequently, every parallelogram
formed by the acute-angled triangles is equidecomposable with the
parallelogram with the set of vertices $\{0, \zeta^2-\zeta^3,
1-\zeta^2, 1-\zeta^3\}$,
and each parallelogram formed by the
obtuse-angled triangles is equidecomposable with the parallelogram
with the set of vertices $\{0, \zeta^2-\zeta^3, \zeta^4-\zeta^3,
\zeta^2-2\zeta^3+\zeta^4\}$. We get 5 parallelograms of each kind.
We can put all the obtained parallelograms
together to form the parallelogram $F$.
\end{proof}

The parallelogram $F$,
seen as a subset of $\mathcal{C}$, is
the fundamental domain of the group $\langle w_1, w_2\rangle<P$. It is
easy to check that $P/\langle w_1, w_2\rangle$ is isomorphic to
$\mathbb{Z}^2\oplus\mathbb{Z}/5\mathbb{Z}$. The space of orbits
$\mathcal{C}/\langle w_1, w_2\rangle$ is naturally homeomorphic to the
parallelogram $F$ (and hence to the spaces $V'$ and $\mathcal{T}$).

\begin{proposition}
The group $\mathcal{P}$ is isomorphic to the full topological group of
the action of $P/\langle w_1, w_2\rangle$ on the Cantor set
$\mathcal{C}/\langle w_1, w_2\rangle$.
\end{proposition}

\begin{corollary}
The derived subgroup of $\mathcal{P}$ is simple and finitely generated.
\end{corollary}


\begin{thebibliography}{AA99}
\bibitem{bruijn:pen1}
\textsc{de~Brujin, N. G.}
{\it Algebraic theory of {P}enrose's non-periodic tilings of the plane
  {I}}.
Indagationes Mathematicae, 43(1):39--52, 1981.

\bibitem{bruijn:pen2}
\textsc{de~Brujin, N.G.}
{\it Algebraic theory of {P}enrose's non-periodic tilings of the plane
  {II}}.
Indagationes Mathematicae, 43(1):53--66, 1981.

\bibitem{C} \textsc{Chou, C.}, {\it Elementary amenable groups} \textrm{Illinois J. Math. 24 (1980), 3, 396--407.}

\bibitem{day:semigroups}  \textsc{Day, M.,} {\it Amenable semigroups}, Illinois J. Math., 1 (1957), 509--544.

\bibitem{grigorchuk:milnor_en}
\textsc{Grigorchuk, R.,} {\it Milnor's problem on the growth of groups}, Sov.
  Math., Dokl, 28 (1983), 23--26.

\bibitem{j-trees}\textsc{Juschenko, K.}, {\it Non-elementary amenable subgroups of automata groups.} arXiv preprint arXiv:1504.00610 (2015).


\bibitem{J}\textsc{Juschenko, K.}, {\it Amenability}. Book in preparation. \hfil {\tt http://www.math.northwestern.edu/\string~juschenk/book.html}.


\bibitem{JM}\textsc{Juschenko, K., Monod, N.}, {\it Cantor systems, piecewise translations
  and simple amenable groups}.  Annals of Mathematics, Volume 178, Issue 2, 775 -- 787, 2013


\bibitem{JNS}\textsc{Juschenko, K., Nekrashevych, V.,  de la Salle, M.}, {\it   Extensions of amenable groups by recurrent groupoids}. \textrm{arXiv:1305.2637}.

\bibitem{JS}\textsc{Juschenko, K., de la Salle, M.}, {\it Invariant means of the wobbling groups}.  \textrm{ arXiv preprint arXiv:1301.4736 (2013).}

\bibitem{JMMS}\textsc{Juschenko, K., Matte Bon, N., Monod, N., de la Salle, M.}, {\it Extensive amenability and an application to interval exchanges.} arXiv preprint arXiv:1503.04977.



\bibitem{Matui}\textsc{Matui, H.}, {\it Some remarks on topological full groups of Cantor minimal systems,}   \textrm{Internat. J. Math. 17 (2006), no. 2, 231--251.}

\bibitem{matui1}\textsc{Matui, H.}, {\it Topological full groups of one-sided shifts of finite type,}   \textrm{Journal fur die reine und angewandte Mathematik (Crelles Journal) (2013).}


\bibitem{matui2}\textsc{Matui, H.}, {\it Homology and topological full groups of \'etale groupoids on totally disconnected spaces,}   \textrm{Proceedings of the London Mathematical Society 104.1 (2012): 27-56.}

\bibitem{nek:fullgr}
\textsc{Nekrashevych, V.}
{\it Simple groups of dynamical origin},
(preprint arXiv:1511.08241), 2015.


\bibitem{von-neumann1}
\textsc{von Neumann, J.,} {\it  Zur allgemeinen Theorie des Masses},  Fund. Math., vol 13 (1929), 73-116.

\bibitem{penrose}
\textsc{Penrose, R.,}
{\it Pentaplexity},
\newblock In F.~C. Holrayd and R.~J. Wilson, editors, {\em Geometrical
  Combinatorics}, pages 55--65. Pitman, London, 1984.

\end{thebibliography}
\end{document}